\documentclass{amsart}
\usepackage{latexsym}
\usepackage{amsfonts}
\usepackage{graphicx}
\usepackage{amsmath}

\newtheorem{lemma}{Lemma}[section]
\newtheorem{proposition}{Proposition}[section]
\newtheorem{theorem}{Theorem}[section]
\newtheorem{corollary}{Corollary}[section]

\theoremstyle{definition}
\newtheorem{remark}{Remark}[section]
\newtheorem{example}{Example}[section]
\newtheorem{definition}{Definition}[section]
\begin{document}
\email{shtrakov@aix.swu.bg}
\address{Department of Computer Science \\
South-West University,  2700 Blagoevgrad, Bulgaria}
\title[Essential variables and positions in terms]
 {Essential variables and positions in terms}
\author[Sl. Shtrakov]{Slavcho Shtrakov}
\urladdr{http://home.swu.bg/shtrakov}
\date{}
\begin{abstract}
 The paper deals with $\Sigma-$composition of terms, which
allows  us to extend the derivation rules in formal deduction of
identities.
 The
 concept of essential
variables and essential positions of terms  with respect to a set of
identities is a key step in the simplification of the process of
formal deduction.  $\Sigma-$composition of terms is defined as
replacement between $\Sigma$-equal terms. This composition induces
$\Sigma R-$deductively closed sets of identities.  In analogy to
balanced identities we introduce and investigate $\Sigma-$balanced
identities for a given set  of identities $\Sigma$.

\end{abstract}
\maketitle

\section{Introduction}
Let ${\mathcal{F}}$ be any finite set, the elements of which are
called $ operation\ symbols. $ Let $\tau:{\mathcal{F}}\to N$ be a
mapping into the non-negative integers; for $f\in{\mathcal{F}},$ the
number $\tau(f)$ will denote the \emph{arity } of the operation
symbol $f.$ The pair $(\mathcal{F},\tau)$ is called a \emph{type} or
\emph{signature}. If it is obvious what the set ${\mathcal{F}}$ is,
we will write ``$ type\ \tau$". The set of symbols of arity $p$ is
denoted by ${\mathcal{F}}_p.$

 Let $X$  be a finite set of variables, and let $\tau$
be a type with the set of operation symbols ${\mathcal
F}=\cup_{j\geq 0}{\mathcal F}_j.$ The set $W_{\tau}(X)$ of \emph{ 
terms\ of\ type\ $\tau$ } with\ variables\ from\ $X $ is the
smallest set such that

\begin{enumerate}
\item[(i)] $X\cup\mathcal{F}_0\subseteq W_\tau(X)$;

\item[(ii)] if $f $ is an $n-$ary operation symbol and
$t_1,\ldots,t_{n}$ are terms, then the ``string" $f(t_1\ldots
t_{n})$ is a term.
\end{enumerate}

 An algebra
 ${\mathcal A} = \langle A; \mathcal{F}^{\mathcal A}\rangle$
  of type $\tau$ is a pair consisting of a set $A$ and an
  indexed  set $\mathcal{F}^{\mathcal{A}}$ of operations, defined on $A$.
  If $f\in\mathcal{F}$, then $f^\mathcal{A}$ denotes a $\tau(f)$-ary operation on the
set $A.$
   We denote by ${\mathcal{A}} lg(\tau)$ the class of all algebras
   of type $\tau$. If $s, t \in W_{\tau}(X)$,
    then the pair $s \approx t$ is called an identity of type
$\tau$. $Id(\tau)$ denotes the set of all identities of type
$\tau.$

An identity $t \approx s\in Id(\tau)$ is satisfied
 in the algebra ${\mathcal A}$, if the
term operations $t^{\mathcal A}$ and $s^{\mathcal A}$, induced by
the terms $t$ and $s$ on the algebra ${\mathcal A}$ are equal, i.e.,
 $t^\mathcal{A}= s^\mathcal{A}.$
     In this case
    we write ${\mathcal A}  \models t \approx s$ and if $\Sigma$ is a set
of identities of type $\tau$, then ${\mathcal A}  \models \Sigma$
means that ${\mathcal A}  \models t \approx s$ for all $ t \approx
s\in\Sigma$.

Let $\Sigma$ be a set of identities. For $t,s\in W_\tau(X)$ we write
$\Sigma\models t\approx s$ if, given any algebra $\mathcal{A}$,\quad
$\mathcal{A}\models \Sigma\ \Rightarrow\ \mathcal{A}\models t\approx
s.$

The operators $Id$ and $Mod$ are defined
     for classes of algebras $K \subseteq \mathcal{A} lg(\tau)$ and
     for sets of identities $\Sigma \subseteq Id(\tau)$ as
     follows

\begin{align*}
Id (K)&:= \{t \approx s \mid {\mathcal A} \in K~  \Rightarrow  {\mathcal A} \models t \approx s\},\ \text{ and}\\
Mod(\Sigma)&:= \{ {\mathcal A} \mid  \  t \approx s \in  \Sigma~ \Rightarrow  {\mathcal A} \models t \approx s\}.
\end{align*}
 The set of fixed points
 with respect to the closure operators
 $Id Mod$  and
 $ Mod Id$
 form complete lattices ${\mathcal L}(\tau)$ and ${\mathcal E}(\tau)$ of all
varieties of type $\tau$ and of all equational theories (logics) of
type $\tau$.

 In  \cite{bur}
deductive closures of sets of identities are used to describe some
elements of these lattices. We will apply the concept of
 $\Sigma-$compositions of terms  to
study the lattices ${\mathcal L}(\tau)$ and ${\mathcal E}(\tau)$. We
use the concept of essential variables, as defined in
 \cite{sht} and therefore we consider such variables  with respect to a given set of
 identities,  which is a
 fully invariant congruence.

In Section \ref{sec2} we investigate the concept of
$\Sigma-$essential variables and positions.  The fictive
(non-essential) variables and positions are used to simplify the
deductions of identities in
 equational theories. We introduce $\Sigma-$composition
of terms for a given set $\Sigma$ of identities.

In Section \ref{sec3}  we describe the  closure operator $\Sigma R$
in the set of all identities of a given type, which generate
extensions of  fully invariant congruences.
    The varieties which satisfy $\Sigma
R-$closed sets are fully invariant congruences and they are called
stable. The stable varieties are compared to solid ones
\cite{den50,gra,tay2}.

In Section \ref{sec4} we introduce and study   $\Sigma-$balanced
identities and prove that $\Sigma-$balanced property is closed under
$\Sigma R$-deductions.

\section{Compositions of terms}\label{sec2}

If $t$ is a term, then the set $var(t)$ consisting of those elements
of $X$ which occur in $t$ is called the set of
 \emph{input variables (or variables)} for $t$.  If
$t=f(t_1,\ldots,t_n)$ is a non-variable term, then $f$ is the \emph{ 
root symbol (root)} of $t$ and we will write $f=root(t).$
 For a term  $t\in W_\tau(X)$ the set   $Sub(t)$ of its
subterms
 is defined as follows:
if $t\in X\cup \mathcal{F}_0$, then $Sub(t)=\{t\}$ and if
$t=f(t_1,\ldots,t_{n})$, then $Sub(t)=\{t\}\cup
Sub(t_1)\cup\ldots\cup Sub(t_n).$

The $depth$ of a term $t$ is defined inductively: if $t\in
X\cup\mathcal{F}_0$ then $Depth(t)=0;$ \ and  if
$t=f(t_1,\ldots,t_{n})$,
then 
 $Depth(t)=max \{Depth(t_1),\ldots,Depth(t_{n})\} +1.$
\begin{definition}%
\label{d-indkompGeneral} Let $r,s,t\in W_\tau(X)$ be three terms of
type $\tau$. By $t(r \leftarrow s)$ we will denote the term,
obtained by simultaneous replacement of  every occurrence of $r $ as
a
 subterm of $t$ by $s$.  This term is called
 the
\emph{inductive composition } of the terms $t$ and $s $, by $r $. In
particular,
\begin{enumerate}
\item[(i)] $t(r\leftarrow s)=t$ if $r\notin
Sub(t)$;

\item[(ii)] $t(r\leftarrow s)=s$ if $t=r$,  and

\item[(iii)] $t(r\leftarrow s)=f(t_1(r\leftarrow
s),\ldots,t_n(r\leftarrow s))$, if
$t=f(t_1,\ldots,t_n)$ and $r\in Sub(t)$, $r\neq t$.
\end{enumerate}
\end{definition}
If $r_i\notin Sub(r_j)$ when $i\neq j$, then $t(r_1\leftarrow
s_1,\ldots,r_m\leftarrow s_m)$ means the inductive composition of
$t,r_1,\ldots,r_m$ by $s_1,\ldots,s_m$. In the particular case when
$r_j=x_j$ for $j=1,\ldots,m$ and $var(t)=\{x_1,\ldots,x_m\}$ we will
briefly write $t(s_1,\ldots,s_m)$ instead of $t(x_1\leftarrow
s_1,\ldots,x_m\leftarrow s_m)$.

Any term can be regarded  as a tree with nodes labelled as the
operation symbols and its leaves   labelled as variables or nullary
operation symbols. Often the tree of a term is presented by a
diagram of the corresponding term  as it is shown by  Figure~\ref{f1}.

 Let $\tau$ be a type and $\mathcal{F}$ be its set of operation
symbols. Denote by $maxar=\max\{\tau(f)\mid f\in\mathcal{F}\}$ and
$N_\mathcal{F}:=\{m\in N \mid m\leq maxar\}$.  Let $N_\mathcal{F}^*$
be the set of all finite strings over $N_\mathcal{F}.$ The set
$N_\mathcal{F}^*$ is naturally ordered by $p\preceq q \iff p$\ is a
prefix of $q.$ The Greek letter $\varepsilon$, as usual denotes the
empty word (string) over $N_\mathcal{F}.$

To distinguish between different occurrences of the same operation
symbol in a term $t$ we assign to each operation symbol a position,
i.e., an element of a given set. Usually positions are finite
sequences (strings) of natural numbers. Each position is assigned to
a node of the tree diagram of $t$, starting with the empty sequence
$\varepsilon$ for the root and using the integers $j$, $1\leq j\leq
n_i$ for the $j$-th branch of an $n_i$-ary operational symbol $f_i$.
So, let the position $p=a_1a_2\ldots a_s\in N^*_\mathcal{F}$ be
assigned to a node of $t$ labelled by the $n_i$-ary operational
symbol $f_i$. Then the position assigned to the $j$-th child of this
node is $a_1a_2\ldots a_sj$. The set of positions of a term $t$ is
denoted by $Pos(t)$ and it is  illustrated by Example \ref{ex1}.

Thus we have $Pos(t)\subseteq N_{\mathcal{F}}^*$.

  Let $t\in W_\tau(X)$ be a term of type $\tau$ and let $sub_t:Pos(t)\to
Sub(t)$  be the function  which maps each position in a term $t$ to
the subterm of $t$, whose root node occurs at that position.

\begin{definition}%
\label{d-kompos3}
 Let $t,r\in W_\tau(X)$ be two terms of type
$\tau$ and $p\in Pos(t)$ be a position in $t.$  The positional
composition of $t$ and $r$ on $p$ is the term $s:=t(p;r)$ obtained
from $t$ by replacing   the term $sub_t(p)$ by $r$  on the position
$p$, only.
\end{definition}
\begin{example}~\label{ex1}\rm
Let $\tau=(2)$, $ t=f(f(x_1,f(f(f(x_1,x_2),x_2),x_3)),x_4)$
 and
$u=f(x_4,x_1)$. The  positions of $t$ and $u$ are written on their
nodes  in Figure \ref{f1}.
 Then the positional composition of $t$ and $u$ on the position $121\in Pos(t)$ is  $ t(121;
u)=f(f(x_1,f(f(x_4,x_1),x_3)),x_4)$ and
$sub_t(121)=f(f(x_1,x_2),x_2)$.
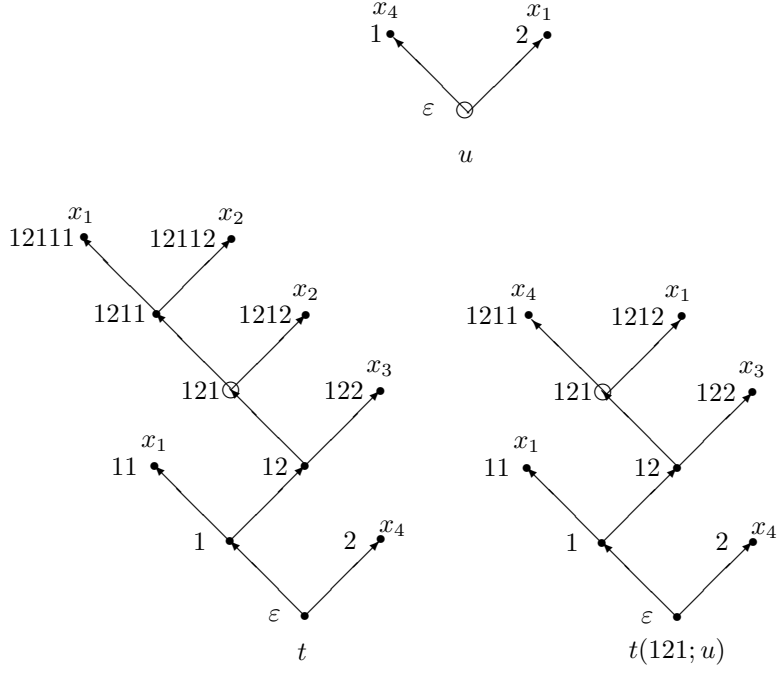
\begin{figure}
\unitlength 1mm 
\linethickness{0.4pt}
\ifx\plotpoint\undefined\newsavebox{\plotpoint}\fi 
\begin{picture}(99.87,90.33)(0,0)
\put(38.73,29.73){\vector(1,1){10}}
\put(88.22,29.55){\vector(1,1){10}}
\put(28.85,39.85){\vector(1,1){10}}
\put(18.98,49.98){\vector(1,1){10}}
\put(78.73,39.2){\vector(1,1){10}}
\put(60.34,76.67){\vector(1,1){10}}
\put(28.73,19.74){\vector(1,1){10}}
\put(78.22,19.56){\vector(1,1){10}}
\put(28.73,19.74){\vector(-1,1){10}}
\put(78.22,19.56){\vector(-1,1){10}} \put(18.73,29.74){\circle*{1}}
\put(68.22,29.56){\circle*{1}} \put(48.73,39.73){\circle*{1}}
\put(98.22,39.55){\circle*{1}} \put(38.85,49.85){\circle*{1}}
\put(88.82,49.67){\circle*{1}} \put(70.96,87.15){\circle*{1}}
\put(28.98,59.98){\circle*{1}} \put(38.73,29.73){\circle*{1}}
\put(88.22,29.55){\circle*{1}} \put(28.85,39.85){\circle{2}}
\put(78.34,39.67){\circle{2}} \put(59.96,77.15){\circle{2}}
\put(18.98,49.98){\circle*{1}} \put(68.47,49.8){\circle*{1}}
\put(50.09,87.28){\circle*{1}} \put(28.73,19.74){\circle*{1}}
\put(78.22,19.56){\circle*{1}}
\put(34.73,29.73){\makebox(0,0)[cc]{$12$}}
\put(84.22,29.55){\makebox(0,0)[cc]{$12$}}
\put(24.85,39.85){\makebox(0,0)[cc]{$121$}}
\put(74.34,39.67){\makebox(0,0)[cc]{$121$}}
\put(55.96,77.15){\makebox(0,0)[rc]{$\varepsilon$}}
\put(14.09,49.98){\makebox(0,0)[cc]{$1211$}}
\put(63.58,49.8){\makebox(0,0)[cc]{$1211$}}
\put(48.21,87.28){\makebox(0,0)[cc]{$1$}}
\put(24.73,19.74){\makebox(0,0)[cc]{$1$}}
\put(74.22,19.56){\makebox(0,0)[cc]{$1$}}
\put(18.73,32.74){\makebox(0,0)[cc]{$x_1$}}
\put(68.22,32.56){\makebox(0,0)[cc]{$x_1$}}
\put(14.73,29.74){\makebox(0,0)[cc]{$11$}}
\put(64.22,29.56){\makebox(0,0)[cc]{$11$}}
\put(43.84,39.73){\makebox(0,0)[cc]{$122$}}
\put(93.33,39.55){\makebox(0,0)[cc]{$122$}}
\put(33.44,49.85){\makebox(0,0)[cc]{$1212$}}
\put(82.93,49.67){\makebox(0,0)[cc]{$1212$}}
\put(67.56,87.15){\makebox(0,0)[cc]{$2$}}
\put(22.39,59.98){\makebox(0,0)[cc]{$12112$}}
\put(48.73,42.73){\makebox(0,0)[cc]{$x_3$}}
\put(98.22,42.55){\makebox(0,0)[cc]{$x_3$}}
\put(38.85,52.85){\makebox(0,0)[cc]{$x_2$}}
\put(88.34,52.67){\makebox(0,0)[cc]{$x_1$}}
\put(69.96,90.15){\makebox(0,0)[cc]{$x_1$}}
\put(28.98,62.98){\makebox(0,0)[cc]{$x_2$}}
\put(38.73,9.86){\vector(1,1){10}}
\put(88.22,9.68){\vector(1,1){10}}
\put(38.73,9.86){\vector(-1,1){10}}
\put(88.32,9.5){\vector(-1,1){10}}
\put(38.73,29.99){\vector(-1,1){10}}
\put(88.22,29.81){\vector(-1,1){10}}
\put(28.86,40.11){\vector(-1,1){10}}
\put(18.98,50.24){\vector(-1,1){10}}
\put(78.9,39.28){\vector(-1,1){10}}
\put(60.52,76.76){\vector(-1,1){10}} \put(38.73,9.86){\circle*{1}}
\put(88.22,9.68){\circle*{1}}
\put(34.73,9.86){\makebox(0,0)[cc]{$\varepsilon$}}
\put(84.22,9.68){\makebox(0,0)[cc]{$\varepsilon$}}
\put(44.73,19.86){\makebox(0,0)[cc]{$2$}}
\put(94.22,19.68){\makebox(0,0)[cc]{$2$}}
\put(50.38,21.39){\makebox(0,0)[cc]{$x_4$}}
\put(99.87,21.21){\makebox(0,0)[cc]{$x_4$}}
\put(3.71,60.1){\makebox(0,0)[cc]{$12111$}}
\put(9.02,63.11){\makebox(0,0)[cc]{$x_1$}}
\put(9.37,60.28){\circle*{1}} \put(48.79,20){\circle*{1}}
\put(98.32,19.62){\circle*{1}}
\put(67.88,52.86){\makebox(0,0)[cc]{$x_4$}}
\put(49.5,90.33){\makebox(0,0)[cc]{$x_4$}}
\put(38.36,5){\makebox(0,0)[cc]{$t$}}
\put(60.1,71){\makebox(0,0)[cc]{$u$}}
\put(88.21,5){\makebox(0,0)[cc]{$t(121;u)$}}
\end{picture}

  \caption{Positional composition of terms}\label{f1}
\end{figure}

\end{example}

\begin{remark}\label{rem-assoc}
The positional composition has the following properties:
\begin{enumerate}

\item[1.] 
If $\langle\langle p_1,p_2\rangle,\langle t_1,t_2\rangle\rangle$ is
a pair with $p_1\not\prec p_2\  \&\  p_2\not\prec p_1$, then
$$t(p_1,p_2;t_1,t_2)=t(p_1;t_1)(p_2;t_2)=t(p_2;t_2)(p_1;t_1);$$

\item[2.] If $S=\langle p_1,\ldots,p_m\rangle$ and $T=\langle t_1,\ldots,t_m\rangle$
with  $$(\forall
 p_i,p_j\in S)\    (i\neq j \Rightarrow p_i\not\prec p_j\  \&\  p_j\not\prec p_i)$$
 and $\pi$ is a permutation of the set $\{1,\ldots,m\}$, then
$$t(p_1,\ldots,p_m;t_1,\ldots,t_m)=
t(p_{\pi(1)},\ldots,p_{\pi(m)};t_{\pi(1)},\ldots,t_{\pi(m)}).$$

\item[3.] If \ $t,s,r\in W_\tau(X)$, \ $p\in Pos(t)$ and\ $q\in Pos(s)$,
then $t(p;s(q;r))=t(p;s)(pq;r)$.

\item[4.] Let $s,t\in W_\tau(X)$ and $r\in Sub(t)$ be  terms of type
$\tau$. Let $\{ p_1,\ldots,p_m\}=\{p\in Pos(t)\mid sub_t(p)=r\}$.
Then we have
$$t(p_1,\ldots,p_m;s):=t(p_1; s)(p_2; s),\ldots,(p_m; s)=t(r\leftarrow s),$$
\end{enumerate}
 which shows that any inductive
composition can be represented as  a positional one. On the other
side there are  examples of  positional compositions which can not
be realized as inductive compositions.
\end{remark}

\begin{definition}
\label{d-closReplac} Let $\Sigma\subseteq Id(\tau)$, $t\in
W_\tau(X_n)$ be an $n-$ary term of type $\tau$, $\mathcal{A}=\langle
A,\mathcal{F}\rangle$ be an algebra of type $\tau$ and let $x_i\in
var(t)$ be a variable which occurs in $t.$

 (i) \cite{sht} The
variable $x_i$ is called \emph{essential}  for $t$ with respect to
 the algebra $\mathcal{A}$ if there are $n+1$ elements
$a_1,\ldots,a_{i-1},a,b,a_{i+1},\ldots,a_n\in A$  such that
$$t^\mathcal{A}(a_1,\ldots,a_{i-1},a,a_{i+1},\ldots,a_n)\neq
t^\mathcal{A}(a_1,\ldots,a_{i-1},b,a_{i+1},\ldots,a_n).$$ The set of
all essential variables for $t$ with respect to  $\mathcal{A}$ will
be denoted by $Ess(t,\mathcal{A})$. $Fic(t,\mathcal{A})$ denotes the
set of all variables in $var(t)$, which are not essential with
respect to $\mathcal{A}$, called fictive ones.

(ii) 
\label{d-EssSig}  A variable $x_i$ is said to be \emph{ 
$\Sigma-$essential} for a term $t$  if there is an algebra
$\mathcal{A}$, such that $\mathcal{A}\models \Sigma$ and $ x_i\in
Ess(t,\mathcal{A}).$ The set of all $\Sigma-$essential variables for
$t$ will be denoted by $Ess(t,\Sigma).$ If a variable is not
$\Sigma-$essential for $t$, then it is called \emph{$\Sigma-$fictive}
for $t$. $Fic(t,\Sigma)$ denotes the set of all $\Sigma-$fictive
variables for $t.$
\end{definition}
\begin{proposition}\label{p-SigmaEss}
If $\Sigma_1\subseteq\Sigma_2\subseteq Id(\tau)$, $t\in W_\tau(X)$
and $x_i\in Ess(t,\Sigma_2)$, then $x_i\in Ess(t,\Sigma_1)$.
\end{proposition}

\begin{theorem}
\label{t-EssSig} Let $t\in W_\tau(X)$ and $\Sigma\subseteq
Id(\tau)$.
 A variable $x_i$ is $\Sigma-$essential for $t$
 if and only if there is a term $r$ of type $\tau$ such that $r\neq
 x_i$ and
$\mathcal{A}\not\models t\approx t(x_i\leftarrow r)$  for some
algebra $\mathcal{A}\in\mathcal{A} lg(\tau)$ with
$\mathcal{A}\models\Sigma.$
\end{theorem}
\begin{proof} \  Let  $t\in W_\tau(X_n)$ for some
$n\in N$ and let $\mathcal{A}\in\mathcal{A} lg(\tau)$ be an algebra
for which $\mathcal{A}\models\Sigma$ and $x_i\in
Ess(t,\mathcal{A})$.
 Then from Lemma
3.5  of \cite{sht} it follows that $\mathcal{A}\not\models t\approx
t(x_i\leftarrow x_{n+1}).$ Hence $\mathcal{A}\not\models t\approx
t(x_i\leftarrow r)$ with $r=x_{n+1}$.

Conversely, let us assume that there is a term $r$, $r\neq x_i$ of
type $\tau$ with $\mathcal{A}\not\models t\approx t(x_i\leftarrow
r)$ for an algebra $\mathcal{A}\in\mathcal{A} lg(\tau)$ with
$\mathcal{A}\models\Sigma.$

 Let $m\in N$
be a natural number for which $r\in W_\tau(X_m)$. So, there are
$m+n$ values
$a_1,\ldots,a_{i-1},a_i,a_{i+1},\ldots,a_n,b_1,\ldots,b_m\in A$
 such that
$r^\mathcal{A}(b_1, \ldots,b_m)\neq a_i$ and $$
t^\mathcal{A}(a_1,\ldots,a_{i-1},a_i,a_{i+1},\ldots,a_n) \neq
t^\mathcal{A}(a_1,\ldots,a_{i-1},r^\mathcal{A}(b_1,
\ldots,b_m),a_{i+1},\ldots,a_n).$$
 The last inequality shows that
$x_i\in Ess(t,\mathcal{A}).$ Hence $x_i$ is $\Sigma-$essential for
$t$.\end{proof}
\begin{corollary}
\label{c-ess1} If $t\approx s\in\Sigma$ and $x_i\in Fic(t,\Sigma)$,
then for each term $r\in W_\tau(X)$, we have $\Sigma\models
t(x_i\leftarrow r)\approx s.$
\end{corollary}
\begin{corollary}
\label{c-ess2} A variable $x_i$ is $\Sigma-$essential for $t\in
W_\tau(X_n)$  if and only if $x_i$ is essential for $t$ with respect
to any $Mod(\Sigma)$-free algebra with at least $n+1$ free
generators.
 \end{corollary}

\begin{corollary}\label{c-EssId2}
\label{c-EssId} Let $\Sigma\subseteq Id(\tau)$ be a set of
identities of type $\tau$ and $t\approx s\in\Sigma$. If a variable
$x_i$ is $\Sigma-$fictive  for $t$, then it is fictive for $s$ with
respect to  each algebra $\mathcal{A}\in Mod(\Sigma)$.
\end{corollary}

    The concept of
$\Sigma-$essential positions is a natural extension of
$\Sigma-$essential variables.

\begin{definition}
\label{d-EssPos1}Let  $\mathcal{A}=\langle A,\mathcal{F}\rangle$ be
an algebra of type $\tau$, $t\in W_\tau(X_n)$, and let $p\in Pos(t)$.

 (i) If $x_{n+1}\in
Ess(t(p;x_{n+1}),\mathcal{A})$, then the
  position $p\in Pos(t)$ is called
\emph{essential} for $t$  with respect to   the algebra
$\mathcal{A}$. The set of all essential positions for  $t$  with
respect to  $\mathcal{A}$ is denoted by $PEss(t,\mathcal{A}).$ When
a position $p\in Pos(t)$ is not essential for $t$ with respect to
$\mathcal{A}$, it is called \emph{fictive} for $t$ with respect to
$\mathcal{A}$. The set of all fictive positions with respect to
$\mathcal{A}$ is denoted by $PFic(t,\mathcal{A}).$

(ii) If $x_{n+1}\in
Ess(t(p;x_{n+1}),\Sigma)$, then the position $p\in Pos(t)$ is called
\emph{$\Sigma-$ essential} for $t$.  The set of $\Sigma-$essential
positions for  $t$ is denoted by $PEss(t,\Sigma).$ When a position
is not $\Sigma-$essential for $t$ it is called \emph{ 
$\Sigma-$fictive}. $PFic(t,\Sigma)$ denotes the set of all
$\Sigma-$fictive positions for $t.$

The set of $\Sigma$-essential subterms of $t$ is defined as follows:
$SEss(t,\Sigma):=\\ \{sub_t(p) \mid  $ $\ p\in PEss(t,\Sigma)\}$.
$SFic(t,\Sigma)$ denotes the set $SFic(t,\Sigma):=Sub(t)\setminus
SEss(t,\Sigma)$.
\end{definition}
So, $\Sigma$-essential subterms of a term are subterms which
occur at a $\Sigma$-essential position. Since one subterm can occur
 at more than one position in a term, and can occur in both $\Sigma$-essential
 and non-$\Sigma$-essential positions, we note that a subterm is $\Sigma$-essential
 if it occurs at least once in a $\Sigma$-essential position, and $\Sigma$-fictive otherwise.

\begin{example}\rm
\label{e-EssPoss1} Let $\tau=(2)$ and let
$t=f(f(x_1,x_2),f(f(x_1,x_2),x_3))$. Let us consider the variety
$RB=Mod(\Sigma)$ of rectangular bands, where
\begin{equation*}~\label{eq2} \Sigma=\{f(x_1,f(x_2,x_3)) \approx f(f(x_1,x_2),x_3) \approx
f(x_1,x_3),\ f(x_1,x_1) \approx x_1\}.\end{equation*}
 It is not difficult to see that the $\Sigma-$essential
positions and subterms of $t$ are
\begin{center}
\begin{tabular}{l}
$PEss(t,\Sigma)=\{\varepsilon,1,11,2,22\}, $\\
$SEss(t,\Sigma)=\{t,f(x_1,x_2),x_1,f(f(x_1,x_2),x_3),x_3\}$\\
$PFic(t,\Sigma)=Pos(t)\setminus PEss(t,\Sigma)= \{12,21,211,212\},$\\ $ SFic(t,\Sigma)=\{x_2\}.$ 
\end{tabular}\end{center}
The $\Sigma-$essential and $\Sigma-$fictive positions of $t$ are
represented by large and small black circles, respectively in Figure
\ref{f-EssPos}. Note that $|PFic(t,\Sigma)|>|SFic(t,\Sigma)|$. This
is because there is one subterm, $f(x_1,x_2)$, which occurs more
than once, once each in an essential and non-essential position, so
that $|Pos(t)|> |Sub(t)|$.
\end{example}

\begin{figure}[h]
\unitlength=1.00mm \special{em:linewidth 0.4pt}
\linethickness{0.4pt}
\begin{picture}(64.00,45.00)
\put(35.00,10.00){\circle*{2.00}}
\put(35.00,1.00){\makebox(0,0)[cc]{$ t$}} \
\put(35.00,10.00){\line(-3,2){15.00}}
\put(35.00,10.00){\line(3,2){15.00}}
\put(20.00,20.00){\circle*{2.00}}
\put(50.00,20.00){\circle*{2.00}}

\put(50.00,20.00){\line(-1,1){10.00}}
\put(50.00,20.00){\line(1,1){10.00}}
\put(20.00,20.00){\line(-1,2){5.00}}
\put(20.00,20.00){\line(1,2){5.00}}
\put(15.00,30.00){\circle*{2.00}}
\put(25.00,30.00){\circle*{1.00}}
\put(60.00,30.00){\circle*{2.00}}
\put(40.00,30.00){\circle*{1.00}}
\put(40.00,30.00){\line(-1,2){5.00}}
\put(40.00,30.00){\line(1,2){5.00}}
\put(35.00,40.00){\circle*{1.00}}
\put(45.00,40.00){\circle*{1.00}}
\put(17.00,20.00){\makebox(0,0)[cc]{$1$}}
\put(23.00,20.00){\makebox(0,0)[cc]{$f$}}
\put(37.00,30.00){\makebox(0,0)[cc]{21}}
\put(43.00,30.00){\makebox(0,0)[cc]{$f$}}
\put(64.00,29.00){\makebox(0,0)[cc]{$x_3$}}
\put(31.50,40.00){\makebox(0,0)[cc]{211}}
\put(45.00,44.00){\makebox(0,0)[cc]{$x_2$}}
\put(25.00,34.00){\makebox(0,0)[cc]{$x_2$}}
\put(16.00,34.00){\makebox(0,0)[cc]{$x_1$}}
\put(12.00,30.00){\makebox(0,0)[cc]{11}}
\put(22.00,30.00){\makebox(0,0)[cc]{12}}
\put(36.00,45.00){\makebox(0,0)[cc]{$x_1$}}
\put(41.00,40.00){\makebox(0,0)[cc]{212}}
\put(54.00,18.00){\makebox(0,0)[cc]{$f$}}
\put(39.00,8.00){\makebox(0,0)[cc]{$f$}}
\put(56.00,30.00){\makebox(0,0)[cc]{22}}
\put(47.00,20.00){\makebox(0,0)[cc]{2}}
\put(31.00,8.00){\makebox(0,0)[cc]{$\varepsilon$}}
\end{picture}

  \caption{$\Sigma-$essential positions for $t$ from Example \ref{e-EssPoss1}.}
\label{f-EssPos}
\end{figure}
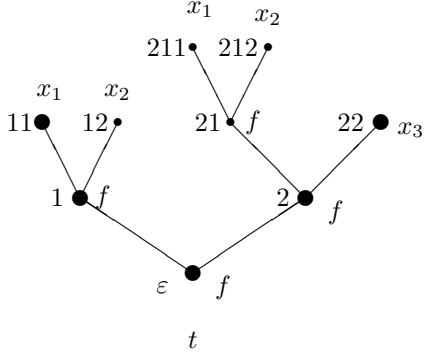

\begin{theorem}
 \label{t-EssPos2}
If $p\in PEss(t,\Sigma)$, then each position $q\in Pos(t)$ with
$q\preceq p$ is $\Sigma-$essential for $t$.
\end{theorem}
\begin{proof}\ Let $sub_t(q) = s$ and $sub_t(p)= r$. Now, $q\preceq
p$ implies that $r\in Sub(s)$ and  $Sub(r)\subset Sub(s).$ Let $n$
be a natural number such that $t\in W_\tau(X_n).$

 From $p\in
PEss(t,\Sigma)$ it follows that there is a term $v\in W_\tau(X)$
for which $v\neq x_{n+1}$ and
$$\Sigma\not\models t\approx t(p;x_{n+1})(x_{n+1}\leftarrow v),\quad\mbox{i.e.,}\quad
\Sigma\not\models t\approx t(p;v).$$

Consequently, there is an algebra $\mathcal{A}=\langle
A,\mathcal{F}\rangle$ of type $\tau$ such that
$$\mathcal{A}\models\Sigma\quad\mbox{and}\quad t^{\mathcal{A}}\neq t(p;v)^{\mathcal{A}}.$$
Let $m$ be a natural number such that $t\in W_\tau(X_m)$ and $v\in
W_\tau(X_m)$.

Let $(a_1,\ldots,a_m)\in A^m$ be  a tuple such that
$$t^{\mathcal{A}}(a_1,\ldots,a_n)\neq t(p;v)^{\mathcal{A}}(a_1,\ldots,a_m).$$

Let $u\in W_\tau(X_m)$ be the term $u=s(q';v),$ where
 $p=q\circ q'$ and $q'\in Pos(s).$ Hence we have
$\Sigma\models t(p;v)\approx t(q;u)$ and
$$t^{\mathcal{A}}(a_1,\ldots,a_n)\neq t(q;u)^{\mathcal{A}}(a_1,\ldots,a_m).$$
Consequently  $t^{\mathcal{A}}\neq t(q;u)^{\mathcal{A}}$, i.e.,
$$\Sigma\not\models t\approx t(q;x_{n+1})(x_{n+1}\leftarrow u)$$
and $q\in PEss(t,\Sigma)$.
 \end{proof}

 \begin{corollary}
 \label{c-EssPos2}
If $q\in PFic(t,\Sigma)$, then each position $p\in Pos(t)$ with
$q\preceq p$ is $\Sigma-$fictive for $t$.
\end{corollary}

\begin{theorem}
\label{t-EssPos1} Let $t\in W_\tau(X)$ be a term of type $\tau$
and let $\Sigma\subseteq Id(\tau)$ be a set of identities of type
$\tau.$ If $p\in PFic(t,\Sigma)$,
 then
 $\Sigma\models t\approx t(p;v),$
for each term $v\in W_\tau(X).$
\end{theorem}
\begin{proof}\  Let $p\in PFic(t,\Sigma)$ and let us suppose that the
theorem is false. Then  there is a term $v\in W_\tau(X)$ with $v\neq
sub_t(p)$, such that $\Sigma\not\models t\approx t(p;v).$ Let
$sub_t(p)=r$ and let $n$ be a natural number, such that $v, t\in
W_\tau(X_n)$. Then
$$t(p;v)=t(p;x_{n+1})(x_{n+1}\leftarrow v)\quad\mbox{ and }\quad
t=t(p;r)=t(p;x_{n+1})(x_{n+1}\leftarrow r).$$ Our supposition shows
that
$$\Sigma\not\models t\approx t(p;v)\ \iff\ \Sigma\not\models
t(p;x_{n+1})(x_{n+1}\leftarrow r)\approx
t(p;x_{n+1})(x_{n+1}\leftarrow v).$$

Hence there is an algebra $\mathcal{A}=\langle A,\mathcal{F}\rangle$
and $n+2$ elements $a_1,\ldots,a_n,a,b$ of $A$ such that
$$(t(p;x_{n+1}))^{\mathcal{A}}(a_1,\ldots,a_n,a)\neq (t(p;x_{n+1}))^{\mathcal{A}}(a_1,\ldots,a_n,b),$$
where $a=r^{\mathcal{A}}(a_1,\ldots,a_n)$ and
$b=v^{\mathcal{A}}(a_1,\ldots,a_n).$

 This means that $x_{n+1}\in Ess(t(p;x_{n+1}),\Sigma)$. Hence
$p\in PEss(t,\Sigma)$, which is a contradiction. \end{proof}
\begin{corollary}\label{c-EssPos1}
If $p\in Pos(t)$ is a $\Sigma-$fictive position for $t$, then $p$ is
fictive for $t$ with respect to  each algebra $\mathcal{A}$ with
$\mathcal{A}\models\Sigma.$
\end{corollary}

\begin{corollary}\label{c-EssPos3}
If $p\in PEss(t,\Sigma)$,  $t\in W_\tau(X_n)$, then $p$ is essential
for $t$ with respect to  each $Mod(\Sigma)-$free algebra with at least
$n+1$ free generators.
\end{corollary}

If $\Sigma\models t\approx s$ and $s\in Sub(t)$ is a proper subterm
of $t$, one might expect that the positions of $t$ which are
``outside" of $s$ have to be $\Sigma-$fictive. To see that this is
not true, we consider the set of operations $\vee,\wedge$ and $\neg$
with type $\tau:=(2,2,1)$. Let $\Sigma$ be the set of identities
satisfied in a Boolean algebra. Then it is easy to prove that if
$t=x_1\wedge(x_2\vee\neg(x_2))$, then we have $\Sigma\models
t\approx x_1$, but $PEss(t,\Sigma)=Pos(t).$

Now, we are going to generalize composition of terms and to describe
the  corresponding deductive systems.

Let $\Sigma$ be a set of identities of type $\tau$. Two terms $t$
and $s$ are called \emph{$\Sigma$-equivalent} (briefly, \emph{$\Sigma$-equal}) if $\Sigma\models t\approx s$.

\begin{definition}\label{d-SigmaComp}
Let $t,r,s\in W_\tau(X)$ and $\Sigma S_r^t=\{v\in Sub(t) \mid 
\Sigma\models r\approx v \}$ be the set of all subterms of $t$ which
are $\Sigma-$equal to $r$.

 Term
$\Sigma-$composition of $t$ and $r$ by $s$ is defined as follows
\begin{enumerate}
 \item[(i)]$t^\Sigma( r\leftarrow s)=t$\ \  if\  \ $\Sigma S_r^t=\emptyset$;

\item[(ii)] $t^\Sigma( r\leftarrow s)=s$\ \  if\ \  $\Sigma\models t\approx
r$,  and

\item[(iii)] $t^\Sigma(r\leftarrow s)=f(t_1^\Sigma(r\leftarrow
s),\ldots,t_n^\Sigma(r\leftarrow s)),$\ \  if\ \
$t=f(t_1,\ldots,t_n)$.
\end{enumerate}
\end{definition}

Let $\Sigma P_r^t=\{p\in Pos(t) \mid sub_t(p)\in \Sigma S_r^t\}$ be
the set of all positions of subterms of $t$ which are $\Sigma-$equal
to $r$. Let $P_r^t=\{p_1,\ldots,p_m\}$ be the set of  all the
minimal elements in $\Sigma P_r^t$ with respect to the ordering
$\prec$ in the set of positions, i.e., $p\in P_r^t$ if for each $q\in
P_r^t$ we have $q\not\prec p$. Let $ r_j=sub_t(p_j)$ for
$j=1,\ldots,m$. Clearly,
$$t^\Sigma(r\leftarrow s)= t(P_r^t;s).$$
\begin{example}~\label{ex3}\rm
Let us consider the set $\Sigma$ of identities satisfied in the
variety $RB$ of rectangular bands (see Example \ref{e-EssPoss1}).
Let $r=f(x_1,x_2)$ and let the terms $t$ and $u$ be the same as in
Example \ref{ex1}. Then we have $$\Sigma S_r^t=\{f(x_1,x_2),
f(f(x_1,x_2),x_2)\},\ \Sigma P_r^t=\{1211,121\}, P_r^t=\{121\}$$ and
$t^\Sigma(r\leftarrow u)=t(121;u)$ (see Figure \ref{f1}).

 So, the term
$t^\Sigma(r\leftarrow u)$ is the term obtained from $t$ by replacing
$r$ by $u$ at any minimal positions whose subterm is $\Sigma$-equal
to $r$, where minimality refers to the order $\preceq$ on the set of
positions.
\end{example}

\begin{proposition}\label{p-Sigma}
If $\Sigma\models r\approx v$, and $u,t\in W_\tau(X)$ then:
\begin{enumerate}

\item[(i)] $\Sigma\models t^\Sigma(u\leftarrow u)\approx t$;

\item[(ii)]  $P^t_r=P^t_v$;

\item[(iii)] $\Sigma\models t^\Sigma(r\leftarrow u)\approx t^\Sigma(v\leftarrow u)$.
\end{enumerate}

\end{proposition}
\begin{proof}\ 
(i) \ If $t^\Sigma(u\leftarrow u)= t$, then the proposition
is obvious. Let us assume that $t^\Sigma(u\leftarrow u)\neq t$.
Hence $\Sigma P^t_u\setminus P^t_u=\{q_1,\ldots,
q_k\}\neq\emptyset$. Let $P^t_u=\{p_1,\ldots, p_m\}$ and $p_i\in
P^t_u$. If $p_i\not\prec q_j$ for all $j\in\{1,\ldots,k\}$, then
since $\Sigma\models sub_t(p_i)\approx u$ and $D_5$ we obtain
 $\Sigma\models t(p_i;u)\approx t.$
 If $p_i\prec q_j$, for some $j\in\{1,\ldots,k\}$,
then we have $\Sigma\models sub_t(p_i)\approx sub_t(q_j)\approx u$.
Since $t=t(p_i;sub_t(p_i))=t(p_i;sub_t(p_i))(q_j;sub_t(q_j))$ we obtain   $\Sigma\models t(p_i;u)\approx t.$
Finally, we have $\Sigma\models t^\Sigma(u\leftarrow u)\approx t$.

(ii) and (iii) are clear.
\end{proof}

\begin{corollary}\label{c-Sigma1}
\begin{enumerate}

\item[(i)] $P^{t^\Sigma(u\leftarrow u)}_u=\Sigma P^{t^\Sigma(u\leftarrow u)}_u$;

\item[(ii)] $t^\Sigma(u\leftarrow u)^\Sigma(u\leftarrow
v)=t^\Sigma(u\leftarrow u)(u\leftarrow v)$ for any term $v\in
W_\tau(X)$.
\end{enumerate}

\end{corollary}
 \begin{proposition}\label{p-PFic}
If $\Sigma\models t\approx s$ and $\Sigma\models r\approx v$, then
$$P^t_r\subseteq PFic(t,\Sigma)\quad\iff\quad P^s_v\subseteq PFic(s,\Sigma).$$
\end{proposition}
Next we consider a deductive system, which is based on the
$\Sigma$-compositions of terms.

\section{Stable varieties and globally invariant congruences}\label{sec3}

 Our next goal is to introduce deductive  closures
 on the subsets of $Id(\tau)$ which generate
elements of the lattices ${\mathcal L}(\tau)$ and ${\mathcal
E}(\tau)$.
 These closures are  based on two concepts -
  satisfaction of an identity by a variety and   deduction
of an identity.

\begin{definition}
\label{d-closDeduct}\cite{bur} A set $\Sigma$ of identities of type
$\tau$ is $D-$deductively closed if it satisfies the following
axioms (some authors call them ``deductive rules", ``derivation
rules", ``productions", etc.):
\begin{enumerate}

\item[$D_1$] {(reflexivity)} $t\approx t\in \Sigma$ for
each term $t\in W_\tau(X)$;

\item[$D_2$] {(symmetry)}  $(t\approx s\in \Sigma)\  \Rightarrow\
s\approx t\in \Sigma$;

\item[$D_3$]  {(transitivity)} $(t\approx s\in \Sigma)\  \&\ (s\approx
r\in \Sigma)\  \Rightarrow\  t\approx r\in \Sigma$;

\item[$D_4$] {(variable inductive substitution)} \\ $(t\approx s\in
\Sigma)\ \&\ (r\in W_\tau(X))\   \Rightarrow\ t(x\leftarrow
r)\approx s(x\leftarrow r)\in\Sigma$;

\item[$D_5$] {(term positional replacement)} \\
$(t\approx s\in \Sigma)\ \&\ (r\in W_\tau(X))\ \&\ (sub_r(p)=t)\
\Rightarrow\ r(p;s)\approx r\in\Sigma$.
\end{enumerate}
\end{definition}
For any set of identities $\Sigma$ the smallest $D-$deductively
closed set containing $\Sigma$ is called the $D-$closure of $\Sigma$
and it is denoted by $D(\Sigma).$

Let $\Sigma$ be a set of identities of type $\tau.$ For $t\approx
s\in Id(\tau)$ we say $\Sigma\vdash t\approx s$ (``$\Sigma$
$D-$proves $t\approx s$")  if there is a sequence of identities
($D-$deduction) $t_1\approx s_1,\ldots,t_n\approx s_n$, such that
each identity belongs to $\Sigma$ or is a result of applying any of
the derivation rules $D_1-D_5$  to previous identities in the
sequence and the last identity $t_n\approx s_n$ is $t\approx s.$

According to \cite{bur},  $\Sigma\models t\approx s$ if and only if
$\ t\approx s\in D(\Sigma)$ and the closure $D(\Sigma)$ is a
 fully invariant congruence for each set $\Sigma$ of identities  of a given
type. It is known that
    there exists a variety
$V\subset \mathcal{A} lg(\tau)$ with $Id(V)=\Sigma$ if and only if
$\Sigma$ is  a fully invariant congruence (Theorem 14.17
\cite{bur}).

Using properties of the essential variables and positions we can
divide the  rules $D_4$ and $D_5$ into four rules which distinguish
between operating with essential or fictive objects in the
identities.
\begin{proposition}\label{p-EssDeriv}
A set $\Sigma$ is $D-$deductively closed if it satisfies rules
$D_1,D_2,D_3$ and
\begin{enumerate}

\item[$D_4'$] {(essential variable inductive substitution)} \\
$(t\approx s\in \Sigma)\ \&\ (r\in W_\tau(X))\ \&\ (x\in
Ess(t,\Sigma))\ \Rightarrow\ t(x\leftarrow r)\approx s(x\leftarrow
r)\in\Sigma$;

\item[$D_4''$] {(fictive variable inductive substitution)} \\
$(t\approx s\in \Sigma)\ \&\ (r\in W_\tau(X))\ \&\ (x\in
Fic(t,\Sigma))\ \Rightarrow\ t(x\leftarrow r)\approx s\in\Sigma$;

\item[$D_5'$] {(essential positional term replacement)} \\
$(t\approx s\in \Sigma)\ \&\  (sub_r(p)=t,\ p\in PEss(r,\Sigma))\
\Rightarrow\ r(p;s)\approx r\in\Sigma;$

\item[$D_5''$] {(fictive positional term replacement)} \\
$(t,s,r\in W_\tau(X))\ \&\ (sub_r(p)=t,\ p \in PFic(r,\Sigma))\
\Rightarrow\ r(p;s)\approx r\in\Sigma$.
\end{enumerate}
\end{proposition}
We will say that a  set $\Sigma$ of identities is complete if
$D(\Sigma)=Id(\tau)$. It is clear that if $\Sigma$ is a complete
set, then $Mod(\Sigma)$ is a trivial variety.

For fictive positions in terms and   complete sets of identities,
we have: $$\Sigma\ \text{is complete} \ \iff (\forall t\in
W_\tau(X)) \ \ Pos(t)=PFic(t,\Sigma).$$

Let $\Sigma$ be a  non-complete set  of identities. Then from
Theorem \ref{t-EssSig} and Theorem \ref{t-EssPos1}, it follows that
when applying the rules $D_4''$ and $D_5''$,  we can skip these
steps in the deduction process, without any reflection  on the
resulting identities. Hence, if $\Sigma$ is  a non-complete set of
identities, then  $\Sigma$ is $D-$deductively closed if it satisfies
the rules $D_1,D_2,D_3,D_4',D_5'.$

In order to obtain new elements in the lattices $\mathcal{L}(\tau)$
and $\mathcal{E}(\tau)$, we have to extend the derivation rules
$D_1-D_5$.

\begin{definition}\label{d-TSR}
A set   $\Sigma$ of identities  is $\Sigma R$-deductively closed if
it satisfies the rules $D_1,D_2,D_3,D_4$ and
\begin{enumerate}

\item[$\Sigma R_1$] \emph{($\Sigma$  replacement)}
\begin{center}
\begin{tabular}{l}$
 (t\approx s, r\approx v, u\approx w\in
\Sigma)
                \&\ (r\in SEss(t,\Sigma))\ \&\ (v\in SEss(s,\Sigma))$\\
  $ 
\Rightarrow t^\Sigma(r\leftarrow u)\approx
s^\Sigma(v\leftarrow w)\in\Sigma.$\end{tabular}
\end{center}
\end{enumerate}
\end{definition}

For any set of identities $\Sigma$ the smallest $\Sigma
R-$deductively closed set containing $\Sigma$ is called $\Sigma
R-$closure of $\Sigma$ and it is denoted by $\Sigma R(\Sigma).$

 Let $\Sigma$
be a set of identities of type $\tau.$ For $t\approx s\in Id(\tau)$
we say $\Sigma\vdash_{\Sigma R} t\approx s $ ($``\Sigma$ $\Sigma R
$-proves $t\approx s"$) if there is a sequence of identities
$t_1\approx s_1,\ldots,t_n\approx s_n$, such that each identity
belongs to $\Sigma$ or is a result of applying any of the derivation
rules
   $D_1, D_2, D_3,  D_4$ or $\Sigma R_1$
  to previous identities
in the sequence and the last identity $t_n\approx s_n$ is
$t\approx s.$

Let $t\approx s$ be an identity and $\mathcal{A}$ be an algebra of
type $\tau$. $\mathcal{A}\models_{\Sigma R} t\approx s$ means that
$\mathcal{A}\models t^\Sigma(r\leftarrow v) \approx
s^\Sigma(r\leftarrow v)$ for every $r\in SEss(t,\Sigma)\cap
SEss(s,\Sigma)$ and $v\in W_\tau(X)$.

Let $\Sigma$ be a set of identities. For $t,s\in W_\tau(X)$ we say
$\Sigma\models_{\Sigma R} t\approx s$ (read: ``$\Sigma$ $\Sigma
R-$yields $t\approx s$") if, given any algebra $\mathcal{A}$,
$$\mathcal{A}\models_{\Sigma R} \Sigma\quad\Rightarrow\quad \mathcal{A}\models_{\Sigma R} t\approx s.$$

\begin{theorem}\label{p-TCclosure}
$\Sigma R$ is a  closure operator in  the set $Id(\tau)$, i.e.,
\begin{enumerate}

\item[(i)] $\Sigma\subseteq \Sigma R(\Sigma)$;

\item[(ii)] $\Sigma_1 \subseteq \Sigma_2 \Rightarrow
\Sigma_1 R(\Sigma_1)\subseteq \Sigma_2 R(\Sigma_2)$;

\item[(iii)] $\Sigma R(\Sigma R(\Sigma))=\Sigma R(\Sigma)$.
\end{enumerate}
\end{theorem}
The following lemma is clear.
\begin{lemma}\label{l-4.1.}
 For each set $\Sigma\subseteq Id(\tau)$ and for each identity $t\approx s\in Id(\tau)$
 we have
\\
\hspace*{.05cm}  $\Sigma\vdash_{\Sigma R} t\approx s \iff \Sigma
R(\Sigma)\vdash t\approx s.$
\end{lemma}
\begin{theorem}\label{t-TCcompleteness}
(The Completeness Theorem for $\Sigma R$-Equational Logic) Let
$\Sigma\subseteq Id(\tau)$ be a set of identities and $t\approx s\in
Id(\tau)$. Then
$$\Sigma\models_{\Sigma R} t\approx s \iff
\Sigma\vdash_{\Sigma R} t\approx s.$$
\end{theorem}
\begin{proof} \   The implication $\Sigma\vdash_{\Sigma R} t\approx
s\Rightarrow \Sigma\models_{\Sigma R} t\approx s$ follows by
$\Sigma\vdash_{\Sigma R} t\approx s\Rightarrow t\approx s\in \Sigma
R(\Sigma)$ since we have used only properties under which $\Sigma
R(\Sigma)$ is closed, i.e.,  under $D_1, D_2, D_3, D_4$ and $\Sigma
R_1.$

For the converse of this, let us note that  for $t\in W_\tau(X)$ we
have $\Sigma\vdash_{\Sigma R} t\approx t$ and if $ t\approx
s\in\Sigma$ then $\Sigma\vdash_{\Sigma R} t\approx s$.

If $\Sigma\vdash_{\Sigma R} t\approx s$, then there is a formal
$\Sigma R-$deduction $t_1\approx s_1,\ldots,t_n\approx
s_n\quad\mbox{of}\quad t\approx s.$ But then $t_1\approx
s_1,\ldots,t_n\approx s_n,\ \  s_n\approx t_n$ is a $\Sigma
R-$deduction of $s\approx t.$

If $\Sigma\vdash_{\Sigma R} t\approx s$ and $\Sigma\vdash_{\Sigma R}
s\approx r$ let $t_1\approx s_1,\ldots,t_n\approx s_n$ be a $\Sigma
R-$deduction of $t\approx s$ and let $\overline s_1\approx
r_1,\ldots,\overline s_k\approx r_k$ be a $\Sigma R-$deduction of
$s\approx r.$ Then
$$t_1\approx s_1,\ldots,t_n\approx s_n,\ \
 \overline s_1\approx r_1,\ldots,\overline s_k\approx r_k,\ \ t_n\approx r_k$$
is a $\Sigma R-$deduction of $t\approx r.$ Hence
$\Sigma\vdash_{\Sigma R} t\approx r.$

Let $\Sigma\models_{\Sigma R} t\approx s$, \ $\Sigma\models_{\Sigma
R} r\approx v$ , $\Sigma\models_{\Sigma R} u\approx w$ and
$\Sigma\models_{\Sigma R} t^{\Sigma}(r\leftarrow u)\approx
s^{\Sigma}(v\leftarrow w)$. Suppose that $\Sigma\vdash_{\Sigma R}
t\approx s$, $\Sigma\vdash_{\Sigma R} r\approx v$ and
$\Sigma\vdash_{\Sigma R} u\approx w$.
 Let $t_1\approx
s_1,\ldots,t_n\approx s_n$, $r_1\approx v_1,\ldots,r_m\approx v_m$
and $u_1\approx w_1,\ldots,u_k\approx w_k$ be $\Sigma R-$deductions
of
$t\approx s$, $r\approx v$ and $u\approx w$. Then 
\begin{gather*}
t_1\approx s_1,\ldots,t_n\approx s_n,\quad
  r_1\approx v_1,\ldots,r_m\approx v_m, \\
u_1\approx w_1,\ldots,u_k\approx w_k,\quad 
t_n^{\Sigma}(r_m\leftarrow u_k )\approx s_n^{\Sigma}(v_m\leftarrow
w_k)
\end{gather*}
 is a $\Sigma R-$deduction of $t^{\Sigma}(r\leftarrow u)\approx s^{\Sigma}(v\leftarrow w)$. Hence
$\Sigma\vdash_{\Sigma R} t^{\Sigma}(r\leftarrow u)\approx
s^{\Sigma}(v\leftarrow w)$.
  \end{proof}
\begin{theorem}
\label{t-TCclosure} For each set of identities $\Sigma$ the closure
 $  \Sigma R(\Sigma)$ is a fully invariant congruence,
 but  $\Sigma R(\Sigma)$ is not in  general equal to $D(\Sigma)$.
 \end{theorem}
\begin{proof}\ \ Let $\Sigma$ be a $\Sigma R-$deductively closed set
of identities. We will prove that $\Sigma$ is a fully invariant
congruence. It has to be shown that $\Sigma$ satisfies the rule
 $D_5$, i.e., if $r\in W_\tau(X)$, $t\approx s
\in\Sigma$ and $p\in Pos(r)$, then $r(p;t)\approx r(p;s)\in\Sigma.$

If $p\notin PEss(r,\Sigma)$, then according to Proposition
\ref{p-EssDeriv} we have $\Sigma\models r(p;v)\approx r(p;w)$ for
all terms $v,w\in W_\tau(X)$.

Let $p\in PEss(r,\Sigma)$ and let $n$ be a natural number such that
$r,t,s\in W_\tau(X_n)$ and let us consider the term
$u=r(p;x_{n+1})$. Clearly, $u\approx u\in\Sigma$, because of $D_1.$
We have
$$u^\Sigma(x_{n+1}\leftarrow v)=u(x_{n+1}\leftarrow v)=u(p;v)$$
for each $v\in W_\tau(X).$
 Now from $\Sigma R_1$ we obtain
$$u^\Sigma(x_{n+1}\leftarrow t)\approx u^\Sigma(x_{n+1}\leftarrow
s)\in\Sigma,\mbox{   i.e.,   } u(x_{n+1}\leftarrow t)\approx
u(x_{n+1}\leftarrow s)\in\Sigma,$$ and
 $r(p;t)\approx r(p;s)\in\Sigma$.

Furthermore,  we will produce a fully invariant congruence $\Sigma$,
which is not $\Sigma R-$ deductively closed. Let us consider the
variety $SG=Mod(\Sigma)$ of semigroups, where $\Sigma=
\{f(x_1,f(x_2,x_3)) \approx f(f(x_1,x_2),x_3) \}$.

From  Theorem 14.17  of \cite{bur} it follows that if $\Sigma$ is a
fully invariant congruence, then  $D(\Sigma)=Id(Mod(\Sigma)).$ Hence
$Id(SG)=D(\Sigma).$
\\ Let $$t=f(f(f(x_1,x_2),x_1),x_2)\quad\mbox{ and }\quad
  s=f(f(x_1,x_2),f(x_1,x_2)).$$
It is not difficult to see that   $\Sigma\models t\approx s$, i.e.,
$t\approx s\in D(\Sigma)$.
 Let us set
$r=v=f(x_1,x_2)$  and $u=w=x_1$. Clearly, for each $z\in W_\tau(X)$
we have $PEss(z,\Sigma)=Pos(z)$ and $r\in SEss(t,\Sigma)$, and $v\in
SEss(s,\Sigma)$. Since $P_r^t=\{11\}$ and $P_s^v=\{1,2\}$,
  we obtain
$$t^\Sigma(r\leftarrow u)=f(f(x_1,,x_1),x_2) \quad\mbox{ and }\quad s^\Sigma(v\leftarrow
w)=f(x_1,x_1).$$
 Hence
$$ \Sigma\not\models t^\Sigma(r\leftarrow u)\approx s^\Sigma(v\leftarrow w).$$
On the other side we have $t^\Sigma(r\leftarrow u)\approx
s^\Sigma(v\leftarrow w)\in \Sigma R(\Sigma)$. Consequently,
$D(\Sigma)$ is a proper subset (equational theory) of $\Sigma
R(\Sigma)$ and $Mod(\Sigma R(\Sigma))$ is a proper subvariety of
$SG$, which contains the variety $RB$ of rectangular bands as a
subvariety, according to Example \ref{ex4}, below.
 \end{proof}

\begin{lemma}\label{l-4.2.}
 For each set $\Sigma\subseteq Id(\tau)$ and for each identity $t\approx s\in Id(\tau)$
 we have
\\
$\Sigma\models_{\Sigma R} t\approx s \iff \Sigma R(\Sigma)\models
t\approx s.$
\end{lemma}
The lemma follows immediately  from Lemma \ref{l-4.1.} and Theorem
\ref{t-TCcompleteness}.

  \begin{definition}\label{d-stableVar}
A set of identities $\Sigma$ is called a globally invariant
congruence if it is $\Sigma R-$deductively closed.

A variety $V$ of type $\tau$ is called stable if $Id(V)$ is $\Sigma
R-$deductively closed, i.e., $Id(V)$ is a globally
invariant congruence.
\end{definition}
Note that when $\Sigma$ is a globally invariant congruence it is
possible to apply substitutions or replacements in any place
(operation symbol) of terms which explains the word ``globally".
\begin{example}\label{ex4}\rm
Now, we will produce a fully invariant congruence $\Sigma$, which is
a globally invariant congruence. Let us consider the variety
$RB=Mod(\Sigma)$ of rectangular
 bands, where
$\Sigma$ is defined as in Example \ref{e-EssPoss1}.

 The set
$\Sigma$ consists of all equations $s \approx t$ such that the first
variable (leftmost) of $s$ agrees with the first variable of $t$,
i.e., $leftmost(t)=leftmost(s)$ and the last variable (rightmost) of
$s$ agrees with the last variable of $t$, i.e.,
$rightmost(t)=rightmost(s)$. It is well known that $\Sigma$ is a
fully invariant congruence and a totally invariant congruence ( see
\cite{den50,gra}). From  Theorem 14.17  of \cite{bur} it follows
that $Id(RB)=D(\Sigma).$

 Let
$t,s,r,v,u,w\in W_\tau(X)$ be six terms such that  $\Sigma\models
t\approx s,\ \Sigma\models r\approx v,\ \Sigma\models u\approx w$,
$r\in SEss(t,\Sigma)$  and $v\in SEss(s,\Sigma)$.

Thus we have 
\begin{alignat*}{2}
leftmost(t)&=leftmost(s), & leftmost(r)&=leftmost(v),\\
leftmost(u)&=leftmost(w),  & rightmost(t)&=rightmost(s),\\
rightmost(r)&=rightmost(v),\quad  &  rightmost(u)&=rightmost(w).
\end{alignat*}

 From  $r\in SEss(t,\Sigma)$ and $v\in SEss(s,\Sigma)$ (see Example
\ref{e-EssPoss1}), we obtain $$leftmost(t)=leftmost(r),\quad
leftmost(s)=leftmost(v)\ \mbox{
 or}$$
  $$rightmost(t)=rightmost(r),\quad 
   rightmost(s)=rightmost(v).$$
Hence
   $$leftmost(t^\Sigma(r\leftarrow u))=leftmost(s^\Sigma(v\leftarrow
 w))\ \mbox{   and   } $$ $$
   rightmost(t^\Sigma(r\leftarrow u))=rightmost(s^\Sigma(v\leftarrow w)).$$
  \end{example}

We are going to compare globally invariant  congruences with the
totally invariant congruences, defined by hypersubstitutions.

In \cite{den50,gra} the  solid varieties are  defined by adding a
new derivation rule which uses the concept of hypersubstitutions.

Let $\sigma:\mathcal{F}\to W_\tau(X)$ be a mapping which assigns to
every operation symbol $f\in\mathcal{F}_n$ an $n-$ary term. Such
mappings are called \emph{hypersubstitutions} (of type $\tau$). If
one replaces every operation symbol $f$ in a given term $t\in
W_\tau(X)$ by the term $\sigma(f)$, then the resulting term
$\hat\sigma[t]$ is the image of $t$ under  the extension
$\hat\sigma$ on the set $W_\tau(X)$. The monoid of all
hypersubstitutions is denoted by $Hyp(\tau).$

Let $\Sigma$ be a set of identities. The hypersubstitution
derivation rule is defined as follows:
\begin{enumerate}

\item[$H_1$] \emph{(hypersubstitution)} \\
$(t\approx s\in \Sigma\ \&\ \sigma\in Hyp(\tau))\ \Rightarrow\
\hat\sigma[t]\approx \hat\sigma[s]\in\Sigma.$
\end{enumerate}

 A set $\Sigma$ is called \emph{$\chi-$deductively closed (hyperequational theory, or totally
invariant congruence)} if it is
 closed with respect to the rules $D_1,D_2,D_3,D_4,D_5$ and $H_1$.
 The $\chi-$closure $\chi(\Sigma)$ of a set $\Sigma$ of identities is defined in a natural
 way and the meaning of $\Sigma\models_\chi$ and
 $\Sigma\vdash_\chi$ is clear.

It is obvious that $D(\Sigma)\subseteq \chi(\Sigma)$  for each set
of identities $\Sigma\subseteq Id(\tau)$.   There are examples of
 $\Sigma$ such that $D(\Sigma)\neq \chi(\Sigma)$, which shows
 that  the corresponding variety
 $Mod(\chi(\Sigma))$ is a proper subvariety of $Mod(D(\Sigma))$.
 A variety $V$ for which $Id(V)$ is $\chi-$deductively closed is
 called \emph{solid} variety of type $\tau$   \cite{den50}.

A more complex closure operator on sets of identities is studied in
\cite{dks1}. This operator  is based on the concept of coloured
terms and multi-hypersubstitutions.

The next proposition deals with the relations between the closure
operators $\Sigma R$ and $\chi$.
\begin{proposition}\label{p-StabSolid}
There exists a  stable variety, which is not a solid variety.
\end{proposition}
\begin{proof}\ Let us consider the type $\tau=(2)$ and
$\Sigma=\{f(f(x_1,x_2),x_1)\approx f(x_1,x_1)\}$. We will show that
$LA=Mod(\Sigma)$ is a stable variety. So, we have to prove that
\begin{equation}\label{eq1}
\Sigma\models t^\Sigma(r\leftarrow u)\approx s^\Sigma(v\leftarrow
w),
\end{equation}
when $\Sigma\models t\approx s$, $\Sigma\models r\approx v$,
$\Sigma\models u\approx w$, $r\in SEss(t,\Sigma)$ and $v\in SEss(s,\Sigma)$.

We will proceed by induction on $Depth(t)$ - the depth of the term
$t$. The case $Depth(t)=0$ is trivial.

 Let $Depth(t)=1$.  If $t=f(x_1,x_2)$ then $\Sigma\models t\approx
 s$ implies $s=f(x_1,x_2)$ and (\ref{eq1}) is satisfied in this
 case. Let us consider the case $t=f(x_1,x_1)$. If $s=f(x_1,x_1)$
 then clearly (\ref{eq1}) holds.

  Let
 $s=f(f(x_1,s_1),x_1)$ for some $s_1\in W_\tau(X)$. Since the
 positions in $s_1$ are $\Sigma$-fictive in $s$ it follows that $v$
 can be one of the terms $x_1$ or $s$. On the other side we have
 $SEss(t,\Sigma)=\{x_1,t\}$. Hence (\ref{eq1}) is satisfied, again.

 Our inductive supposition is that if $Depth(t)<k$ then (\ref{eq1}) is
 satisfied for all $s,r,u,v,w\in W_\tau(X)$ with
  $\Sigma\models t\approx s$, $\Sigma\models r\approx v$,
$\Sigma\models u\approx w$, $r\in SEss(t,\Sigma)$ and $v\in SEss(s,\Sigma)$.

Let $Depth(t)=k>1$ and $Depth(s)\geq k$. Then we have $t=f(t_1,t_2)$
and $s=f(s_1,s_2)$, such that  $t_1$ or $t_2$ is not a variable.

If $\Sigma\models t\approx r$ or $\Sigma\models s\approx v$, then by  Definition \ref{d-SigmaComp} $(ii)$ and
the transitivity $D_3$, it follows  that  (\ref{eq1}) is $\Sigma\models u\approx w$
and we are done.

Next, we assume that $\Sigma\not\models t\approx r$ and
$\Sigma\not\models s\approx v$.

First, let  $t_1\in X$. Then  $s_1=t_1$ and $\Sigma\models
t_2\approx s_2$. Thus,   from the inductive supposition it follows
that (\ref{eq1}) is satisfied.

Second, let $t_1\notin X$. Then we have $s_1\notin X$, also. Hence
 $t=f(f(t_{11},t_{12}),t_2)$, $s=f(f(s_{11},s_{12}),s_2)$
and $\Sigma\models t_2\approx s_2$.

Let $\Sigma\models t_{11}\approx t_2$ and $\Sigma\models
s_{11}\approx s_2$. Then we have $\Sigma\models t \approx
f(t_2,t_2)$ and  $\Sigma\models f(t_2,t_2) \approx f(s_2,s_2)$. On
the other side all positions in $t_{12}$ and $s_{12}$ are
$\Sigma$-fictive in $t$ and $s$, respectively. Thus we have
$$\Sigma\models t^\Sigma(r\leftarrow u)\approx f(t_2^\Sigma(r\leftarrow u),t_2^\Sigma(r\leftarrow u))\ \mbox{
and  }$$ $$\Sigma\models s^\Sigma(v\leftarrow w)\approx
f(s_2^\Sigma(v\leftarrow w),s_2^\Sigma(v\leftarrow w)).$$ Hence
(\ref{eq1}) is satisfied, in this case, again. If $\Sigma\models
t_{11}\approx t_2$ and $\Sigma\not\models s_{11}\approx s_2$ then we
have $\Sigma\models f(t_2,t_2) \approx f(s_1,s_2)$ and
$\Sigma\models s_{1}\approx s_2\approx t_2$. This implies that
(\ref{eq1}) is satisfied, again.

Let $\Sigma\not\models t_{11}\approx t_2$ and $\Sigma\models
t_{1}\approx t_2$. Then we have $\Sigma\models t \approx
f(t_1,t_2)$. Now, we proceed similarly  as in the case
$\Sigma\models t_{11}\approx t_2$ and $\Sigma\not\models
s_{11}\approx s_2$. If $\Sigma\not\models t_{11}\approx t_2$ and
$\Sigma\not\models t_{1}\approx t_2$, then we have
$\Sigma\not\models s_{1}\approx s_2$. Hence $\Sigma\models
t_{1}\approx s_1$ and $\Sigma\models t_{2}\approx s_2$. Again, from
the inductive supposition we prove (\ref{eq1}).

To prove that $LA$ is not a solid variety, let us consider the
following terms $y=f(f(x_1,x_2),x_1)$ and $z=f(x_1,x_1)$. Let
$\sigma\in Hyp(\tau)$ be the hypersubstitution, defined as follows:
$\sigma(f(x_1,x_2)):=f(x_2,x_1)$. It is clear that $\Sigma\models
y\approx z$. On the other side we have
$\hat\sigma[y]=f(x_1,f(x_2,x_1))$ and $\hat\sigma[z]=f(x_1,x_1)$.
Thus, we obtain $\Sigma\not\models \hat\sigma[y]\approx
\hat\sigma[z]$. Hence $LA$ is not a solid variety.
\end{proof}

\begin{remark}~\label{r2}
By analogy, it follows that the variety $$RA=Mod(\{f(x_1,f(x_2,x_1))
\approx f(x_1,x_1)\})$$ is stable, but not solid, also. 

The
varieties of left-zero bands $L0=Mod(\{f(x_1,x_2)\approx x_1\})$ and
of right-zero bands $R0=Mod(\{f(x_1,x_2)\approx x_2\})$ are other
examples of stable varieties, which are not solid ones.

We do not know whether there is a non-trivial solid variety which is
not  stable?
\end{remark}

\section{$\Sigma-$balanced identities and simplification of deductions}\label{sec4}

Regular identities \cite{bur,gra} are identities in which the same
variables occur on each side of the identity. Balanced identities
are identities in which each variable occurs the same number of
times on each side of the identity.

In an analogous way we consider the concept of $\Sigma-$balanced
identities.

Let $t,r\in W_\tau(X)$ be two terms of type $\tau$ and
$\Sigma\subset Id(\tau)$ be a set of identities. $EP^t_r$ denotes
the set of all $\Sigma$-essential positions from $P^t_r$, i.e.,
$EP^t_r=PEss(t,\Sigma)\cap P^t_r$.

\begin{definition}
\label{d-sigbalanced} Let $\Sigma\subset Id(\tau)$.   We will say
that an identity $t\approx s$ of type $\tau$ is $\Sigma-$\emph{ 
balanced} if $| EP^t_q|= | EP^s_q|$ for all $q\in W_\tau(X)$.
\end{definition}
\begin{example}\label{e-balanced}\rm
Let  $\Sigma$ be the set of identities satisfied in the variety $RB$
of rectangular bands (see Example \ref{e-EssPoss1}).

 Let us consider the following three terms
$t=f(f(x_1,x_2),f(x_1,x_3))$,\ \  $s=f(x_1,f(f(x_1,x_2),x_3))$ and
$r=f(f(f(x_1,f(x_3,x_2)),x_3),f(x_1,x_3))$. Clearly, $\Sigma\models
t\approx s$,  $\Sigma\models t\approx r$,
 $SEss(t,\Sigma)=SEss(r,\Sigma)=\{x_1, f(x_1,x_2), f(x_1,x_3),x_3\}$ and $SEss(s,\Sigma)=\{x_1,
f(x_1,x_3),x_3\}$. Thus we have
\\
$EP^t_{x_1}=\{11\},$\ $EP^t_{x_3}=\{22\},$ \
$EP^t_{f(x_1,x_2)}=\{1\},$\ $EP^t_{f(x_1,x_3)}=\{\varepsilon\},$\
$EP^t_t=\{\varepsilon\},$ $EP^s_{x_1}=\{1\},$\ $EP^s_{x_3}=\{22\},$
\ $EP^s_{f(x_1,x_3)}=\{\varepsilon\},$\  $EP^s_s=\{\varepsilon\},$
 and
\ $EP^r_{x_1}=\{111\},$\ $EP^r_{x_3}=\{22\},$ \
$EP^r_{f(x_1,x_2)}=\{11\},$\ $EP^r_{f(x_1,x_3)}=\{\varepsilon\},$
$EP^r_r=\{\varepsilon\}.$ Hence the identity $ t\approx r$ is
$\Sigma-$balanced, but $ t\approx s$ is not $\Sigma-$balanced.
\end{example}
\begin{theorem}
\label{t-sigbalanced1} Let $\Sigma\subset Id(\tau)$ be a set of
 $\Sigma-$balanced identities. If there is a $\Sigma R-$deduction
of $ t\approx s$  with $\Sigma-$balanced  identities,  then
$t\approx s$ is a $\Sigma-$balanced identity of type $\tau.$
\end{theorem}
\begin{proof}\ Let $t,s,r\in W_\tau(X)$ and let $t\approx s$ and
$s\approx r$ be two $\Sigma-$balanced identities of type $\tau$.
Then for each term $q\in W_\tau(X)$ we have  $$|EP^t_q|=|EP^s_q|\quad\mbox{and}\quad  |EP^s_q|=|EP^r_q|.$$ Hence $ |EP^t_q|=|EP^r_q|$
which shows that the identity 
$t\approx r$ is $\Sigma-$balanced, too.

Let $t\approx s$ is a $\Sigma-$balanced identity in $\Sigma
R(\Sigma)$ and let $r\in W_\tau(X)$ be a term with $t\in
SEss(r,\Sigma)$ and $sub_r(p)=t$. We have $\Sigma\models_{\Sigma R}
r(p;s)\approx r$. From Proposition \ref{p-Sigma}, we obtain
$$(EP^t_q=EP^s_q\ \&\ t\in SEss(r,\Sigma))\ \Rightarrow EP^r_q=EP^{r(p;s)}_q$$
for all $q\in W_\tau(X)$. Consequently the identity $r(p;s)\approx
r$ is $\Sigma-$balanced, too.

Let $t\approx s$, $r\approx v$ and $u\approx w$  be
$\Sigma-$balanced identities from $\Sigma R(\Sigma)$. We have to
prove that  the resulting identity $t^\Sigma(r\leftarrow u)\approx
s^\Sigma(v\leftarrow w)$ is $\Sigma-$balanced.

This  will be  done by induction on the depth (also called ``height"
by some authors).

(i) The  basis of induction is  $Depth(t)=1$ (the case $Depth(t)=0$ is
trivial). Let $t=f(x_1,\ldots,x_n)\in W_\tau(X_n)$ and let
$s=g(s_1,\ldots,s_m).$ Hence, if $r=x_i$ for some
$i\in\{1,\ldots,n\}$, then $EP^t_r=\{i\}$  and
$|EP^t_r|=|EP^s_v|=1$.

(ia)  If $\Sigma\models r\approx t$, then $\Sigma\models v\approx s$ and
we have $t^\Sigma(r\leftarrow u)=u$ and $s^\Sigma(v\leftarrow w)=w$.
Hence $t^\Sigma(r\leftarrow u)\approx s^\Sigma(v\leftarrow w)$ is
$\Sigma-$balanced in this case.

(ib) Let $\Sigma\not\models r\approx x_i$ for each $x_i\in X_n$, i.e.,
$r\not\in X_n$ and $\Sigma\not\models r\approx t$. If $r\notin
X_n\cup\{t\}$, then $EP^t_r=EP^s_v=\emptyset$. Thus we have
$$t^\Sigma(r\leftarrow u)=t\ \mbox{ and}\ s^\Sigma(v\leftarrow w)=s$$ and
 the resulting identity $t\approx s$ is $\Sigma-$balanced.

(ic) Let $\Sigma\models r\approx x_i$ for some $x_i\in X_n$ and
$\Sigma\not\models r\approx t$. We have $$t^\Sigma(r\leftarrow
u)=t^\Sigma(x_i\leftarrow u)\ \mbox{and}\ s^\Sigma(v\leftarrow
w)=s^\Sigma(x_i\leftarrow w).$$ Then $EP^{t^\Sigma(x_i\leftarrow
u)}_q=EP^u_q$ and $EP^{s^\Sigma(x_i\leftarrow w)}_q=EP^w_q$ for each
$q\in W_\tau(X)$, i.e., the resulting identity $t^\Sigma(r\leftarrow
u)\approx s^\Sigma(v\leftarrow w)$ is $\Sigma-$balanced, again.

(ii) Let $Depth(t)>1$, $t=f(t_1,\ldots,t_n)$ and $s\in W_\tau(X)$, such
that the identity
 $\Sigma\models t\approx s$ is $\Sigma$-balanced. Suppose that for each $t'\in
SEss(t,\Sigma)$ with $t'\neq t$ the following is true: if $\Sigma\models t'\approx s'$ is
$\Sigma$-balanced identity for some $s'\in W_\tau(X)$, then $
t'^\Sigma(r\leftarrow u)\approx s'^\Sigma(v\leftarrow w)$ is
$\Sigma$-balanced, also.

(iia) Let $\Sigma\models r\approx t$. Then we have
$$EP^{t^\Sigma(x_i\leftarrow u)}_q=EP^u_q\ \mbox{and}\ EP^{s^\Sigma(x_i\leftarrow
w)}_q=EP^w_q$$ for each $q\in W_\tau(X)$ and the resulting identity
is $\Sigma-$balanced in that case, again.

(iib)  Let $\Sigma\not\models r\approx t$ and  $\Sigma\models
r\approx t_i$ for some $i=1,\ldots,n.$ As in the case $(ic)$ it can
be proved that the identity $t^\Sigma(r\leftarrow u)\approx
s^\Sigma(v\leftarrow w)$ is $\Sigma-$balanced.

(iic) Let $\Sigma\not\models r\approx t$, $\Sigma\not\models
r\approx t_i$ for each $i=1,\ldots,n$ and there is $j\in
\{1,\ldots,n\}$ with
$EP^{t_j}_r=\{r_{j1},\ldots,r_{jk_j}\}\neq\emptyset.$ Without loss
of generality assume that all such  $j$ are the natural numbers from
the set  $L=\{1,\ldots,l\}$ with $l\leq n.$

Let $j\in L$. If $\Sigma\models t\approx t_j$ is a $\Sigma-$balanced
identity, then $\Sigma\models s\approx t_j$ is  $\Sigma-$balanced,
also and by our assumption, we have that $t_j^\Sigma(r\leftarrow
u)\approx s^\Sigma(v\leftarrow w)$ is $\Sigma-$balanced identity.
Hence we have $|EP^t_q|=|EP^{t_j}_q|$ for all $q\in W_\tau(X).$ This
implies that the resulting identity is $\Sigma-$ balanced in this
case, also.

If $\Sigma\not\models t\approx t_j$ for all $j\in L$ then since
$t\approx s$ is a $\Sigma-$balanced identity, there are subterms
$s_1,\ldots, s_l$ of $s$ such that  $\Sigma\models t_j\approx s_j$.
 According to our inductive supposition the last
identities are $\Sigma-$balanced. Consequently,
$$|EP^{t_j}_r|=|EP^{s_j}_r|,\ EP^{t}_r=\cup_{j=1}^l EP^{t_j}_r\
\mbox{ and}\ EP^{s}_v=\cup_{j=1}^l EP^{s_j}_v.$$ Hence
$t^\Sigma(r\leftarrow u)\approx s^\Sigma(v\leftarrow w)$ is a
$\Sigma-$balanced identity. \end{proof}

The complexity of the problem of deduction depends on the complexity
of the algorithm for checking when a position of a term is essential
or not with respect to  a set of identities. The complexity of that
algorithm for finite algebras is discussed in \cite{sht}, but it is
based on the full exhaustion of all possible cases.

There should be  a case or cases,  when the process of deduction can
be effectively simplified. This is,  for instance, when a variable
$x$ does not belong to $var(t)$ and $\Sigma\models t\approx s.$
Therefore  we obtain $x\notin Ess(t,\Sigma)\cup Ess(s,\Sigma)$(see
Theorem \ref{t-EssSig}). Then we can skip the rules $D_4''$ and
$D_5''$, according to Proposition \ref{p-EssDeriv}.
 Obviously, it is very easy to
check if $x\in var(t)$ or not.

\end{document}